\documentclass[11pt,reqno]{amsart}
\usepackage{amsmath, amsfonts, epsfig, amssymb, graphics, hyperref, amscd}
\usepackage{graphicx,xypic}
\usepackage[usenames]{color}
\input xy \xyoption{all}
\numberwithin{equation}{section}

\newtheorem{theorem}{Theorem}
\newtheorem{ex}[theorem]{Example}
\newtheorem{prop}[theorem]{Proposition}
\newtheorem{definition}[theorem]{Definition}
\newtheorem{rem}[theorem]{Remark}

\def\S{{\mathbb S}}

\def\Na{{\mathbb N}}
\def\Z{{\mathbb Z}}

\def\R{{\mathbb R}}
\def\K{{\mathbb K}}

\def\Sp{{\mathbf{Sp}}}
\def\SpB{{\mathbf{Sp}^{\H}}}
\def\Set{{\mathbf{Set}}}
\def\SetB{{\mathbf{Set}^{\H}}}
\def\Vec{{\mathbf{Vec}}}
\def\gVec{{\mathbf{gVec}}}

\def\H{{\mathcal{H}}}

\def\p{{\textsf{p}}}
\def\q{{\textsf{q}}}
\def\r{{\textsf{r}}}
\def\s{{\textsf{s}}}
\def\o{{\textsf{o}}}
\def\exp{{\textsf{e}}}
\def\eB{{\textsf{e}_{\H}}}

\def\lB{\ell_{\H}}

\def\DQL{DQ\Lambda}
\def\F{\mathcal{F}}

\def\S{{\mathcal{S}}}

\def\Kb{\overline{K}}

\def\KB{K^{\H}}
\def\KBb{\overline{K^{\H}}}

\def\KBt{\widetilde{K^{\H}}}

\def\n{\bar{n},n}
\def\pmn{\pm n}
\def\pmm{\pm m}
\def\pms{\pm s}
\def\pmt{\pm t}
\def\pmmn_n{[\pm (n+m)]\backslash[\pm n]}
\def\pmmns_mn{[\pm (n+m)]\backslash[\pm n]}
\def\m{\bar{m},m}

\def\Is{(I,\sigma )}
\def\I0{(\bar{I}\cup I,\sigma_0)}
\def\Jt{(J,\tau)}
\def\J0{(\bar{J}\cup J,\sigma_0)}

\def\Im{I/\sigma}

\def\Ss{S,\sigma_S}
\def\Ts{T,\sigma_T}
\def\pmS{\pm S}
\def\pmT{\pm T}
\def\pmU{\pm U}
\def\pmV{\pm V}

\def\sn{s\colon [n]\rightarrow [\bar{n},n]}
\def\seI{s\colon I/\sigma\rightarrow I}

\def\sI0{s\colon I\rightarrow I\cup \bar{I}}
\def\sJ0{s\colon J\rightarrow J\cup \bar{J}}

\def\id{\mathrm{id}}
\def\can{\mathrm{can}}
\def\st{\mathrm{st}}

\def\can{\mathrm{can}}

\def\Sn{\mathbf{\mathcal{S}}_n}
\def\Bn{\mathcal{B}_n}

\def\QSym{\mathrm{QSym}}
\def\DQSym{\mathrm{DQSym}}

\newcommand{\quasishuf}{{\,\begin{picture}(10,5)
\multiput(0,0)(5,0){3}{\line(0,1){5}}%
\put(0,0){\line(1,0){10}}
\put(-0.2,-1){\line(1,0){10.5}}\end{picture}~}}

\def\shuf{{\sqcup\!\sqcup}}

\title[Hyperoctahedral species]{Hyperoctahedral species}
\author[N.~Bergeron]{N.~Bergeron}
\address{Mathematics and
Statistics, York University, 4700 Keele St., Toronto, ON, M5A 4T5, Canada}
\author[P.~Choquette]{P.~Choquette}
\email{bergeron@mathstat.yorku.ca}
\email{philcho@mathstat.yorku.ca}
\subjclass[2000]{16W30; 18D10; 05E10}

\begin{document}
\begin{abstract}
We introduce a new definition for the species of type $B$, or $\H$-species, analog to the classical species (of type $A$), but on which we consider the action of the groups $\Bn$ of signed permutations. We are interested in algebraic structure on these $\H$-species and give examples of Hopf monoids. The natural way to get a graded vector space from a species, given in this paper in terms of functors, will allow us to deepen our understanding of these species. In particular, the image of the classical species $\ell^{\ast}\circ \exp_+$ under a given functor is isomorphic to the combinatorial Hopf algebra $\DQSym$.
\end{abstract}

\maketitle

%{ \parskip=0pt\footnotesize \tableofcontents}
\parskip=8pt

\begin{section}{Introduction}
The theory of species of Joyal \cite{Joyal} has open a lot of interesting problems in combinatorics. In particular, many generalizations of species have been studied, most of them as $\Sn$-modules, as in Bergeron \cite{Bergeron}, Mendez and Nava \cite{Mendez}, but as modules of other groups as well, Joyal and Street \cite{JoyalStreet} for the general linear groups and Hetyei, Labelle and Leroux \cite{HLL} for the hyperoctahedral groups. Using the theory of species as it was formulated in \cite{Joyal}, some have been interested in finding relationships with well-known algebras such as the descent algebra \cite{PR} or combinatorial Hopf algebras \cite{Aguiar}.

In this paper, we want to formulate the definition of $\H$-species, denoted $\SpB$, given in terms of $\H$-sets, on which there is a natural action of the groups $\Bn$. It turns out that these $\H$-species are equivalent to cubical species defined by Hetyei, Labelle and Leroux \cite{HLL}. The former are given as functors from the category of $\H$-sets to the category of vector spaces, where $\H$-sets are sets on which there is an involution without fixed points. 

Our three main goals are to study the algebraic structure of some $\H$-species, construct $\H$-species from usual species and forming graded vector spaces, from $\H$-species. Our first aim will be achieved by endowing the $\H$-species with a twisted multiplication and comultiplication to obtain Hopf monoids. Then we give a well-chosen functor, $\S$, from the usual species to the $\H$-species. It is well-chosen by the fact that it sends the regular representation of $\Sn$ to the regular representation of $\Bn$. Finally, the natural way to form a graded vector space from a species will allow us to see a connection between a well-known Hopf algebra, $\DQSym$ and the usual species $\ell^*\circ \exp_+$, using the functor $\S$ and the functor $\KBt\colon\SpB\rightarrow \gVec$, defined in Section~(\ref{Sec:functors}).

Aguiar and Mahajan \cite{Aguiar} have shown that the graded algebra associated to the species $\ell^*\circ \exp_+$ is isomorphic to the algebra $\QSym$, which has a basis indexed by compositions. On the other hand, the algebra $\DQSym$ of diagonally symmetric functions has a basis indexed by bicompositions. Furthermore, the map from $\DQSym$ to $\QSym$ which sends a bicomposition to a composition can be found using the species theory. It is the natural transformation $\alpha$ in the following diagram 
\begin{equation}\label{Diag:KBt}
\xymatrix{ & \SpB \ar[dr]^{\KBt} \ar@{=>}[d]^{\alpha} \\ \Sp \ar[rr]_{\Kb} \ar[ur]^{\S} & & \gVec}
\end{equation}
where the maps will be defined in Section~(\ref{Sec:functors}).

Acknowledgment: Thanks to Marcelo Aguiar for his many suggestions and interesting problems regarding the $\H$-species.

\end{section}

\section{Premilinaries}
We begin by a brief introduction of species and cubical species. For more details, see \cite{Aguiar},\cite{BLL}, \cite{Joyal} and \cite{HLL}. Recall that $[n]=\{1,2,\dots,n\}$. Let $\Set$ and $\Vec$ be respectively the categories of finite sets with bijections and vector spaces over $\K$ with linear maps. 

\begin{definition}
A species $\p$ is a functor 
$$\p\colon\Set\rightarrow \Vec.$$
So for each finite set $I$, there is a vector space, denoted $\p[I]$, and if we consider the sets $[n]$, $n\geq 0$, we will use the notation, $p[n]$ for $\p[\{1,2,\dots,n\}]$. Also, for each bijection $f:I\rightarrow J$ there is a linear map $\p[f]\colon\p[I]\rightarrow \q[J]$.
%In particular, $\p[\id]=\id_{\p}$ and for any other bijection $g\colon J\rightarrow K$, $\p[g\circ f]=\p[g]\circ\p[f]$.
\end{definition}

The collection of species forms a category, denoted $\Sp$. The species' morphisms $\alpha\colon\p\rightarrow \q$ are natural transformations, i.e. for each finite set $I$, a linear map $\alpha_I\colon\p[I]\rightarrow \q[I]$ such that for all bijection $\sigma\colon I\rightarrow J$ of finite sets the following diagram commutes
\begin{equation}
\begin{CD}
\p[I] @>\alpha_I>> \q[I]\\
@V\p[\sigma] VV  @VV\q[\sigma] V\\
\p[J] @>>\alpha_J> \q[J]
\end{CD}
\end{equation}

Each permutation $\pi\in \Sn$ induces a map
$$\p[\pi]\colon\p[n]\rightarrow \p[n]$$
which makes $\p[n]$ a $\Sn$-module. So a species $\p$ can equally be defined by a sequence $\p[0],\p[1],\p[2],\dots$ of $\Sn$-modules.

Some example of species,
\begin{enumerate}
\item {\bf Exponential species}: $\exp[I]:=\K$, for all $I$,
\item {\bf Linear order species}: $\ell[I]:=\K$-span of all linear order on $I$,
\item {\bf Graph species}: $G[I]:= \K$-span of all graphs on vertex set $I$.
\end{enumerate}

Hetyei, Labelle and Leroux \cite{HLL}, defined a theory of species, the cubical species, for which the hyperoctahedral groups, $\Bn$, are the acting groups. $\Bn$ is the group of signed permutation. A signed permutation $w=w_1\dots w_n$ is a permutation such that some of the integers can be barred. The group $\Bn$ is generated by the transpositions $s_i=(i,i+1)$, $1\leq i \leq n$ and by the signed permutation $s_0=\bar{1}2\dots n$. We give here the essential of the definition of cubical species.

For a finite set $I$, let $E_I$ be the euclidean space of functions $x\colon I\rightarrow \R$ with orthonormal basis $e_i\colon I\rightarrow \R$, defined by
\begin{equation*}
e_i(j)=\begin{cases} 1 &\text{if $i=j$,}\\ 0 & \text{otherwise.} \end{cases}
\end{equation*}
Every element $x\in E_I$, where $x\colon I\rightarrow \R$, can be written as $x=\sum_{i\in I} x(i)e_i$. The category of cubes is the category with objects the Euclidean spaces $E_I$, one for each finite set $I$, and one space $E_{-1}$, for the empty set. Note that $E_n:=E_{\{1,2,\dots,n\}}$. The morphisms are the isometries $\eta\colon E_I\rightarrow E_J$ taking the cube $\Box_I$ into the cube $\Box_J$, where the cube $\Box_I$ is the following convex hull:
$$\Box_I=\mathrm{conv}(\{0\}\cup \{x\in E_I : x(I) \subseteq \{-1,1\}\})$$
Again, let $\Box_n:=\Box_{[n]}$. In fact, the isometries are the bijections taking the set $\{\pm e_i:i\in I\}$ into the set $\{\pm e_j:j\in J\}$. Such an isometry $\eta$ can be written as $(\sigma,\epsilon)$, where $\sigma\colon I\rightarrow J$ is a bijection and $\epsilon$ is a map from $I$ to $\{-1,1\}$. Applied to a basis element $e_i$, it then takes the form $\eta(e_i)=\epsilon(i)\cdot e_{\sigma(i)}$. The cubical morphisms taking the $n$-cube into itself form a group: the hyperoctahedral group $\Bn$.

It can be useful to note that the vertices of $\Box_n$ are functions $x\colon [n]\rightarrow [-1,1]$, where $[-1,1]$ is the real interval, such that $x(i)\in\{-1,1\}$, for all $i\in [n]$. A nonempty face of $\Box_n$ is encoded as a vector $(\Phi(1),\dots,\Phi(n))$, where $\Phi$ is a function $\Phi\colon [n]\rightarrow \{-1,1,\ast\}$. Set $\Phi(i)=1$ or $\Phi(i)=-1$ if every element $x$ of the face is such that $x(i)=1$ or $x(i)=-1$. Otherwise, set $\Phi(i)=\ast$.

A cubical species is a functor from the category of cubes to the category of vector spaces with linear maps. Here are some examples of cubical species:
\begin{ex}

{\bf Uniform species} $E$: \begin{equation} E[E_n]=\begin{cases} \K[E_n] & \text{for $n\geq 0$,}\\ \emptyset & \text{for $n=-1$.}\end{cases}\end{equation}

{\bf Species of faces} $\F$:
\begin{equation}
\F[E_U]=\K[\Phi : \Phi \text{ is a face of $\Box_U$}]
\end{equation}

%{\bf Species of collection of faces} $\Cf$:
%\begin{equation}
%\Cf[E_U]=\K[\{\Phi_1,\dots, \Phi_n\} : n\geq 0, \Phi_i \text{ is a face of } \Box_U]
%\end{equation}

\end{ex}

\section{$\H$-Species}
We have seen that a species is a sequence of $\Sn$-modules. We want to generalize the theory to have a sequence of $\Bn$-modules. Keep in mind that $[\bar{n},n]$ will denote the set $\{\bar{n},\dots,\bar{1},1,\dots,n\}$. When it is more convenient we will use the notation $[\pm n]$ for the same set, where $-i=\bar{i}$ for all $i\in[n]$.

\begin{definition}
An $\H$-set $(I,\sigma)$ is a finite set $I$, together with an involution $\sigma$ on $I$, where $\sigma$ is without fixed points. A bijection of $\H$-sets between $(I,\sigma)$ and $(J,\tau)$, called a $\H$-bijection, is a bijection $f\colon I\rightarrow J$ such that this diagram commutes:
\begin{equation}
\begin{CD}
I @>{f}>> J \\
@V{\sigma}VV  @VV{\tau}V\\
I @>>{f}> J
\end{CD}
\end{equation}
We will write $f\colon \Is\rightarrow \Jt$ for such a $\H$-bijection.
\end{definition}

Let $\SetB$ be the category of $\H$-sets with $\H$-bijections. There is a natural involution $\sigma_0\colon \Na\cup \overline{\Na}\rightarrow \Na\cup \overline{\Na}$, where $\overline{\Na}=\{\bar{1},\bar{2},\dots\}$, that is define, for $i\in\Na$ and $\bar{i}\in\overline{\Na}$ by $\sigma_0(i)=\bar{i}$ and $\sigma_0(\bar{i})=i$. Each time we consider a set $S=\{s_1,\dots,s_k\}\subset \Na$ and its negative $\overline{S}=-S=\{\overline{s_1},\dots,\overline{s_k}\}$, we will endow it with the natural involution $\sigma_0$. We may then omit the involution in this case, i.e. $\overline{S}\cup S=\pm S:=(\overline{S}\cup S,\sigma_0)$.

There is a full subcategory of $\SetB$ composed of the $\H$-sets $([\bar{n},n],\sigma_0)$, where $[\bar{n},n]=[\pm n]=\{\bar{n},\dots,\bar{1},1,\dots,n\}$ and $\sigma_0(i)=\bar{i}$, one for every $n$, called the skeleton of $\SetB$. Indeed, for every $\H$-set $(I,\sigma)$ there is a unique $n$ for which $\pi:(I,\sigma)\rightarrow ([\bar{n},n],\sigma_0)$ is a $\H$-bijection. So we can work uniquely with these $\H$-sets $[\n]$ without loss of generality.

We want to define order-preserving $\H$-bijection. To do that, we need to order the $\H$-sets. The natural order on the sets $[\bar{n},n]$ is $\bar{n}<\dots<\bar{1}<1<\dots <n$. We want to set an order on $I$, for any $\H$-set $(I,\sigma)$, for all $I\subseteq \Z$. Consider the set
\begin{equation}
\Im:=\{\mathrm{max}(i,\sigma(i)) : \forall i\in I\}.
\end{equation}

In particular $[\bar{n},n]/\sigma_0=[n]$. Let $I/\sigma=\{i_1,\dots, i_n\}$ be ordered by the natural order of the integers. Then the overall order on $I$ will be $\sigma(i_n)<\dots<\sigma(i_1)<i_1<\dots<i_n$. For example, let $I=\{-5,-1,2,3,8,9\}$ and $\sigma(-5)=3$, $\sigma(-1)=2$ and $\sigma(8)=9$. Then $I/\sigma=\{2,3,9\}$ and $I=\{8,-5,-1,2,3,9\}$ when ordered. For any other finite set $I$, we pick one order for each couple $(i,\sigma(i))$ and one order for $\Im$.

Two bijections in $\Set$ and two $\H$-bijections in $\SetB$ are particularly important for us in this paper. They are the canonical map, $can$, and the standardization map, $st$. They are defined to be the only order-preserving maps between sets of same cardinality. Explicitly, let $I,J$ in $\Set$ be two sets of same cardinality. Then $$\st\colon I\rightarrow [|I|],\,\can\colon I \rightarrow J.$$
Similarly, for $\H$-sets $\Is$, $\Jt$ and $(\n,\sigma_0)$, where $|I|=|J|=2n$ and $|I/\sigma|=n$, define the $st$ and $can$ maps for $\H$-sets as 
$$\st\colon I \rightarrow [\n],\,\can\colon I\rightarrow J$$
where $\st\circ\sigma=\sigma_0\circ\st$ and $\can\circ\sigma=\tau\circ\can$.

\begin{definition}
A species of type $B$, or $\H$-species, with values in $\Vec$, the category of vector spaces over $\K$ with linear maps, is a functor
$$\p:\SetB\rightarrow \Vec$$
\end{definition}

So there is one vector space $\p[(I,\sigma)]$ for each $(I,\sigma)\in \SetB$ and one linear map $\p[f]$, for each $\H$-bijection $f\colon (I,\sigma)\rightarrow (J,\tau)$. If the context is clear, we may write $\p[I,\sigma]$ instead of $\p[(I,\sigma)]$ and $\p[\bar{n},n]$ instead of $\p[([\bar{n},n],\sigma_0)]$. An arrow $\alpha\colon \p\rightarrow \q$ between two $\H$-species is a natural transformation, i.e. a family of maps
$$\alpha_{I,\sigma}\colon\p[I,\sigma]\rightarrow \q[I,\sigma]$$
one for each $(I,\sigma)\in\SetB$, such that for each $\H$-bijection $f\colon (I,\sigma)\rightarrow (J,\tau)$,
this diagram commutes:

\begin{equation}
\begin{CD}
\p[I,\sigma] @>\alpha_{I,\sigma}>> \q[I,\sigma] \\
@V{\p[f]}VV  @VV{\q[f]}V\\
\p[J,\tau] @>>\alpha_{J,\tau}> \q[J,\tau]
\end{CD}
\end{equation}
Reformulating the diagram as

\begin{equation}
\xymatrix{ & \p[I,\sigma] \ar[rr]^{\alpha_{I,\sigma}} \ar[dd]^<<<<{\p[f]} \ar[dl]_{\p[\sigma]} & & \q[I,\sigma] \ar[dd]^{\q[f]} \ar[dl]^{\q[\sigma]} \\
\p[I,\sigma] \ar[rr]_>>>>>>{\alpha_{I,\sigma}} \ar[dd]^{\p[f]} & & \q[I,\sigma] \ar[dd]_>>>>{\q[f]} \\
& \p[J,\tau] \ar[rr]^>>>>>>{\alpha_{J,\tau}} \ar[dl]^{\p[\tau]} & & \q[J,\tau] \ar[dl]^{\q[\tau]}\\
\p[J,\tau] \ar[rr]_{\alpha_{J,\tau}} & & \q[J,\tau]}
\end{equation}
these equalities must be satisfied:
\begin{equation}
\begin{aligned}
&\q[f]\circ\alpha_{I,\sigma}\circ\p[\sigma]  = \q[f]\circ\q[\sigma]\circ\alpha_{I,\sigma} = \q[\tau]\circ\q[f]\circ\alpha_{I,\sigma}=\\
&\q[\tau]\circ\alpha_{J,\tau}\circ\p[f] = \alpha_{J,\tau}\circ\p[\tau]\circ\p[f]  = \alpha_{J,\tau}\circ\p[f]\circ\p[\sigma]
\end{aligned}
\end{equation}

Denote $\SpB$ the category of species of type $B$, or $\H$-species.

\begin{rem}
The set of $\H$-bijection taking the $\H$-set $[\n]$ to itself form a group, the hyperoctahedral group or the group of signed permutations, $\Bn$. More generally, the set of $\H$-bijection taking $(I,\sigma)$ to itself is the group $B_{I\backslash \sigma}$, which is isomorphic to $\Bn$, if $|I|=2n$.
\end{rem}

\begin{prop}
A $\H$-species $\p$ can be defined as a sequence
$$\p[\emptyset],\p[\bar{1},1],\p[\bar{2},2],\dots$$
of $\Bn$-module and such that morphism of $\H$-species are maps
$$\p[\n]\rightarrow \q[\n]$$
of $\Bn$-modules.
\end{prop}
\begin{proof}[\bf Proof.]
Since $\p$ is a functor, each $\pi\in \Bn$ induces a map
$$\p[\pi]\colon\p[\n]\rightarrow \p[\n]$$
of vector spaces, which turns $\p[\n]$ into a $\Bn$-module. Conversely, any sequence of $\Bn$-modules can be seen as a functor $\p\colon\SetB\rightarrow\Vec$ by using the $\can$ and $\st$ maps.
\end{proof}

%Furthermore, the only order-preseving bijection between a $\H$-set $(I,\tau)$, where $|I|=2n$ and $([\bar{n},n],\sigma_0)$:
%$$\st:(I,\tau)\rightarrow ([\bar{n},n],\sigma_0)$$
%induces an isomorphism of vector spaces
%$$\p[I,\tau]\simeq \p[[\bar{n},n],\sigma_0]$$
%and $\p[I,\tau]$ is a $\Bn$-module using
%$$\p(\pi).\p[I,\tau]=\p(\st^{-1})\circ\p(\pi)\circ\p(\st).(\p[I,\tau])$$

\begin{prop}
The category of cubes is equivalent to the category of $\H$-sets.
\end{prop}
\begin{proof}[\bf Proof.]
Identify a $\H$-set $([\n],\sigma_0)$ with the set of functions $\{\pm e_i: i\in [n]\}$. Then every $\H$-bijection $\tau:([\n],\sigma_0)\rightarrow ([\m],\sigma_0)$ gives rise to an isometry $\phi\colon E_{n}\rightarrow E_{m}$ that takes the cube $\Box_{n}$ into the cube $\Box_{m}$, or in other words it gives rise to a map that takes the set $\{\pm e_i: i\in[n]\}$ to the set $\{\pm e_j :j\in [m]\}$. Conversely, every isometry $\phi\colon E_{n}\rightarrow E_{m}$ that takes the cube $\Box_{n}$ into the cube $\Box_{m}$ can be written as $\phi(e_u)=\epsilon (u)\cdot e_{\sigma(u)}$, for $e_u\in E_{n}$, where $\sigma\colon [n]\rightarrow [m]$ is a bijection and $\epsilon\colon [n]\rightarrow \{-1,1\}$ is a map. Which means that such an isometry is a $\H$-bijection.
\end{proof}

This shows the equivalence between our definition of $\H$-species and the definition of cubical species of Hetyei, Labelle and Leroux \cite{HLL}. 

It can be useful to find the equivalent objects in terms of $\H$-species for the vertices and the faces of the cubes $\Box_n$. For this we need the notion of a section map $\seI$. It is a map such that $s(i)\in \{i,\sigma(i)\}$. In particular, for $s\colon [n]\rightarrow [\pm n]$, $s(i)=\pm i$. For practical reasons, it is useful to denote a section by a list of its images $(s(1),\dots,s(n))$.

We have seen that the vertices of the cube $\Box_n$ are the functions $x\colon [n]\rightarrow [-1,1]$ such that $x(i)\in \{-1,1\}$. Denote a vertex $x$ by the list its images: $x=(x(1),\dots,x(n))$. So each vertex is a section map $s=(s(1),\dots,s(n))$, where $s(i)=i$ if $x(i)=1$ and $s(i)=\bar{i}$ if $x(i)=-1$. The faces of the cubes $\Box_n$ are vector $(\Phi(1),\dots, \Phi(n))$, where $\Phi(i)=1$ or $\Phi(i)=-1$ if every vertex $x$ in the face is such that $x(i)=1$ or $x(i)=-1$, respectively, and $\Phi(i)=\ast$, otherwise. We can identify such a face $(\Phi(1),\dots, \Phi(n))$ with a set of section maps $S=\{\sn\}$ such that for all $s\in S, s(i)=i$ if $\Phi(i)=1$, $s(i)=\bar{i}$ if $\Phi(i)=-1$ and if $\Phi(i)=\ast$, $s(i)$ can either be $i$ or $\bar{i}$. In other words, given a face $\Phi$, let $A=\{a_1,\dots,a_k\}\subset [n]$ be the set $A=\{i:\Phi(i)=1 \text{ or } \Phi(i)=-1\}$ and let $s(A)=(s(a_1),\dots,s(a_k))$, where $s(a_i)=a_i$ if $\Phi(i)=1$ and $s(a_i)=\overline{a_i}$ if $\Phi(i)=-1$. Now given the list $s(A)$, let $\{s(A)^+\}$ be the set all the lists $(s(1),\dots,s(n))$ with $(s(a_1),\dots,s(a_k))$ as a sublist. Then $S=\{s(A)^+\}$.

For example, consider the $\H$-set $([\bar{2},2],\sigma_0)$. The vertices of $\Box_2$ correspond to the section maps: 
$$(1,1),(1,-1),(-1,1),(-1,-1)\leftrightarrow (1,2),(\bar{1},2),(1,\bar{2}),(\bar{1},\bar{2})$$
and the faces of $\Box_2$ correspond to the set of sections maps:
\begin{equation*}
\begin{aligned} 
& (1,1),(1,-1),(-1,1),(-1,-1),(1,\ast),(-1,\ast),(\ast,1),(\ast,-1),(\ast,\ast) \leftrightarrow \\
& \{(1,2)\}, \{(1,\bar{2})\}, \{(\bar{1},2)\}, \{(\bar{1},\bar{2})\}, \{(1,2), (1,\bar{2})\}, \{(\bar{1},2), (\bar{1},\bar{2})\},\\
& \{(1,2),(\bar{1},2)\}, \{(1,\bar{2}),(\bar{1},\bar{2})\}, \{(1,2), (1,\bar{2}), (\bar{1},2), (\bar{1},\bar{2})\}
 \end{aligned}
 \end{equation*}

\begin{ex}
{\bf Exponential species} 

$\eB[\Is]:=\K$, for all $\H$-sets $\Is$. It is the trivial representation.

{\bf Linear order species}
\begin{equation*}
\begin{aligned} 
\lB[\Is]&=\K\text{-span of linear orders on } s(\Im), \text{ for all section maps $\seI$}\\
&=\bigoplus_{\seI} \ell[s(\Im)]\\
\end{aligned}
\end{equation*}
where $\ell$ is the usual species of linear order. For example,
$$\lB[\bar{2},2]=\ell[\{1,2\}]\oplus \ell[\{\bar{1},2\}]\oplus\ell[\{1,\bar{2}\}]\oplus\ell[\{\bar{1},\bar{2}\}]$$
It is the regular representation.

{\bf The species of section maps}
Corresponding to the cubical species of faces:
$$\F[\n]=\K[\{s(A)^+\}: A\subseteq  [n], s\colon A\rightarrow \bar{A}\cup A]$$
%For example,
%\begin{equation}
%\begin{aligned}
%C[[\bar{2}, 2],\sigma_0]=\K[&\{(1,2)\}, \{(\bar{1},2)\}, \{(1,\bar{2})\}, \{(\bar{1},\bar{2})\}, \{(1,2),(1,\bar{2})\}, \{(\bar{1},2),(\bar{1},\bar{2})\}, \\
%& \{(\bar{1},2),(1,2)\}, \{(\bar{1},\bar{2}),(1,\bar{2})\}, \{(1,2),(\bar{1},2),(1,\bar{2}),(\bar{1},\bar{2})\}]
%\end{aligned}
%\end{equation}

%{\bf The species of collection of section maps} Corresponding to the cubical species of collection of faces.
%$$\Cf[\n]=\K[\{ s\in \{s(A_i)^+\}: 1\leq i\leq k \}: k\geq 0,\, A_i\subset [n],\, s\colon A_i\rightarrow \bar{A_i}\cup A_i ]$$
\end{ex}

\section{Product and monoids}

In this section, we set the definitions for monoids, comonoids and Hopf monoids in the category of $\H$-species. For this we need the concept of an $\H$-set partition and of an $\H$-set composition. An $\H$-set partition, $P=\{P_1,\dots, P_k\}$, of an $\H$-set $(I,\sigma)$ is a set of disjoint subsets $P_i$ of $I$ such that $\sigma(P_i)=P_i$, for all $1\leq i\leq k$. We write $P\vdash (I,\sigma)$. An $\H$-set composition, $F=F_1|\dots |F_k$, or an $\H$-composition of $(I,\sigma)$ is an ordered list of disjoint subsets of $I$ such that $\sigma(F_i)=F_i$, for $1\leq i\leq k$. We write $F\models \Is$. A decomposition is a composition with two parts. Each part of an $\H$-set composition $F=F_1|\dots |F_k \models \Is$, or an $\H$-set partition, is itself an $\H$-set. We write $(F_i,\sigma_{F_i})$, where $\sigma_{F_i}$ is the restriction of $\sigma$ to $F_i$. If $F_1|\dots|F_k$ is a composition of $[\n]$, we use $(F_i,\sigma_0)$ for each part.

\begin{prop}
$(\SpB,\cdot)$ is a symmetric monoidal category, with product
\begin{equation}
(\p\cdot \q)[I,\sigma]=\bigoplus_{S|T\models (I,\sigma)} \p[\Ss]\otimes \q[\Ts].
\end{equation}
The unit for this product is
\begin{equation*} \o[(I,\sigma)]=\begin{cases} \K & \text{if $I=\emptyset$,}\\ 0 & \text{otherwise.}\end{cases}\end{equation*}
and the braiding is given by
$$\beta:\p\cdot\q\simeq \q\cdot\p.$$
\end{prop}

The axioms, found in \cite{Mac}, are straightforward to check. The proof that $(\Sp,\cdot)$ is a symmetric monoidal category in \cite{Aguiar} is very similar.

A {\it monoid} in $\SpB$ is a species $\p$ with multiplication map $\mu$ and unit map $\eta$:
$$\mu:\p\cdot\p\rightarrow \p, \eta:\o\rightarrow \p$$
which are associative and unital in the usual sense. There is one linear map for each $\H$-set $(I,\sigma)$ and each $\H$-set decomposition $S|T \models \Is$:
$$\mu_{S,T}:\p[\Ss]\otimes\p[\Ts]\rightarrow \p[\Is]$$
and one linear map for the unit
$$\eta_{\emptyset}:\K\rightarrow \p[\emptyset].$$
A {\it comonoid} in $\SpB$ is a species $\p$ with a comultiplication map $\Delta$ and counit map $\epsilon$:
$$\Delta:\p\rightarrow \p\cdot \p, \epsilon :\p\rightarrow \o$$
which are coassociative and counital in the usual sense. There is one map for each $\H$-set $\Is$:
$$\Delta_{(I,\sigma)}\colon \p[I,\sigma]\rightarrow \bigoplus_{S|T\models \Is} \p[\Ss]\otimes\p[\Ts]$$
and for each decomposition $(S,T)\models \Is$, let
$$\Delta_{S,T}\colon \p[I,\sigma]\rightarrow \p[\Ss]\otimes\p[\Ts]$$
and one linear map for the counit:
$$\epsilon_{\emptyset}:\p[\emptyset]\rightarrow \K,$$
A {\it bimonoid} is a species with monoid and comonoid structures such that $\Delta$ and $\epsilon$ are morphisms of monoids. A {\it Hopf monoid} $\p$ is a bimonoid with a map $s:\p\rightarrow \p$, the antipode, satisfying for each $\H$-set $(I,\sigma)$, the following two equations:
\begin{equation}
\begin{aligned}
\oplus\mu_{S,T}\circ (s_S\otimes id)\circ \oplus\Delta_{S,T}&=\oplus\mu_{S,T}\circ (id\otimes s_T)\circ \oplus\Delta_{S,T} =0\\
\mu_{\emptyset,\emptyset}\circ(\id \otimes s_{\emptyset})\circ \Delta_{\emptyset}&=\eta_{\emptyset}\circ\epsilon_{\emptyset}\\
\end{aligned}
\end{equation}
where, for the first equality, the direct sum is over all decomposition $S|T\models \Is$.

It is interesting to note that as it is the case for bialgebras, if $\p$ is connected, i.e. $\p[\emptyset]=\K$, and $\p$ is a bimonoid then $\p$ is a Hopf monoid. See \cite{Aguiar} for details. For a species $\p$, its dual species $\p^{\ast}$ is defined as
$$\p^{\ast}[I,\sigma]=\p[I,\sigma]^{\ast}$$
where $\p[I,\sigma]^{\ast}$ is the dual vector space of $\p[I,\sigma]$.

\begin{ex}
The $\H$-species $\lB$ is a Hopf monoid. For an $\H$-set $(I,\sigma)$ and a decomposition $S|T$, the multiplication is given by concatenation:
\begin{equation}
\begin{aligned}
\mu^{\lB}_{S,T}\colon \lB[S,\sigma_S]\otimes\lB[T,\sigma_T] &\rightarrow \lB[I,\sigma]\\
l_1\otimes l_2 &\mapsto l_1l_2
\end{aligned}
\end{equation}
where for lists $l_1=l^1_1|\dots |l^i_1$ and $l_2=l^1_2|\dots|l_2^j$, $l_1l_2=l^1_1|\dots |l^i_1|l^1_2|\dots|l_2^j$. The comultiplication is given by the deshuffle of the parts: 
\begin{equation}
\begin{aligned}
\Delta^{\lB}_{(I,\sigma)} \colon \lB[I,\sigma] &\rightarrow \bigoplus_{S|T\models (I,\sigma)} \lB[S,\sigma_S]\otimes\lB[T,\sigma_T] \\
l &\mapsto \sum_{S|T\models \Is} l_{|_S}\otimes l_{|_T}
\end{aligned}
\end{equation}
where for a set $A$ and a list $l=1_1\dots l_k$, $l_{|_A}=l_1\cap A | \dots | l_k\cap A$.

The $\H$-species $\F$ is a Hopf monoid. The multiplication is given by
\begin{equation}
\begin{aligned}
\mu^{\F}_{S,T}\colon \F[S,\sigma_S]\otimes \F[T,\sigma_T] & \rightarrow \F[I,\sigma]\\
\{F_1,\dots,F_k\}\otimes \{G_1,\dots,G_l\} &\mapsto \{F_i\cup G_j:\: 1\leq i\leq k,\, 1\leq j\leq l\}
\end{aligned}
\end{equation}
and the comultiplication is given by
\begin{equation}
\begin{aligned}
\Delta^{\F}_{I,\sigma} \colon \F[I,\sigma] &\rightarrow \bigoplus_{S|T\models (I,\sigma)} \F[S,\sigma_S]\otimes \F[T,\sigma_T]\\
\{H_1,\dots, H_m\} &\mapsto \sum_{S|T\models (I,\sigma)} (H\cap S)\otimes (H\cap T)
\end{aligned}
\end{equation}

For example, for $\{(1),(\bar{1})\}\in \F[(\{\bar{1},1\},\sigma_0)]$, $\{(\bar{2})\}\in \F[(\{\bar{2},2\},\sigma_0)]$ and $\{(1,2),(1,\bar{2})\}\in \F[\bar{2},2]$, we have:
\begin{equation*}
\begin{aligned}
\mu^{\F}(\{(1),(\bar{1})\},\{(\bar{2})\}) =& \{ (1,\bar{2}),(\bar{1},\bar{2})\}\\
\Delta^{\F}(\{(1,2),(1,\bar{2})\}) =&\{(1,2),(1,\bar{2})\}\otimes \emptyset + \{(1)\}\otimes \{(2),(\bar{2})\} + \\
& \{(2),(\bar{2})\}\otimes \{(1)\} + \emptyset\otimes \{(1,2),(1,\bar{2})\}
\end{aligned}
\end{equation*}

\end{ex}

\begin{ex}
$\eB$ is a self-dual Hopf monoid by using the isomorphism $I\mapsto (I)^{\ast}$, where $I$ is the only element in $\eB[(I,\sigma)]$. 

The dual Hopf monoid $\lB^{\ast}$ has multiplication and comultiplication given by the shuffle and the deconcatenation. Let $S|T\models (I,\sigma)$ and $l_1\in \lB^{\ast}[S,\sigma_S]$, $l_2\in \lB^{\ast}[T,\sigma_T]$, $l\in \lB^{\ast}[I,\sigma]$. Then the multiplication is $l_1^{\ast}\otimes l_2^{\ast}\mapsto \sum_{l\in l_1\shuf l_2} l^{\ast}$, where the shuffle $l_1\shuf l_2$ of linear orders $l_1=l_1^1\dots l_1^k$ of a set $S$ and $l_2=l_2^1\dots l_2^l$ of a set $T$, where $S$ and $T$ are disjoint, is defined to be the set of linear order $\{l=l^1\dots l^m\}$ where $l|_S=l_1$ and $l|_T=l_2$. The comultiplication is $l^{\ast}\mapsto \sum_{l=l_i|l_f} l_i^{\ast}\otimes l_f^{\ast}$, where the sum is over all decomposition of $l$ into an initial segment $l_i$ and a final segment $l_f$.
\end{ex}

\section{Composition of species}

The goal of this section is to recall the definition of composition of species (Joyal \cite{Joyal}) and the algebraic structures on these new species (Aguiar \cite{Aguiar}). Let $I$ be a finite set. A set partition $F$ of $I$ is a set of disjoint subsets of $I$, $\{F_1,F_2,\dots, F_n\}$. We write $F\vdash I$. A set composition $G=G_1|G_2|\dots|G_k$ of $I$ is an ordered list of disjoint subsets of $I$. We write $G\models I$. 
Let $F=F_1|\dots|F_k\models S$ and $G=G_1|\dots|G_l \models T$ be two set compositions. The restriction of $F$ to a set $S'\subseteq S$ is the set composition
$$F|_{S'}=F_1\cap S'|\dots|F_k\cap S'$$
where we delete any empty parts. The concatenation of $F$ and $G$ is the set composition of $S\cup T$ defined by
$$F|G=F_1|\dots|F_k|G_1|\dots|G_l.$$

The shuffle of $F$ and $G$, $F\shuf G$, is the set of all set compositions $H=H_1|\dots|H_m$ of $S\cup T$, such that $H|_S=F$, $H|_T=G$ and also $m=k+l$, where $k$ is the number of parts of $F$ and $l$ is the number of parts of $G$. The quasishuffle of $F$ and $G$, $F\quasishuf G$, is the set of all set compositions $H$ of $S\cup T$ such that $H|_S=F$ and $H|_T=G$. In other words, a set composition $H$ is a quasishuffle of $F$ and $G$ if it can be obtained by shuffling the parts of $F$ and $G$ and two adjacent parts $F_i|G_j$ or $G_j|F_i$ can be replaced by $F_i\sqcup G_j$, given that $F_i$ is a part of $F$ and $G_j$ is a part of $G$.

For a composition $G=G_1|\dots|G_l$, let, for any species $\p$, $\p(G)=\p[G_1]\otimes \dots \otimes \p[G_l]$. The composition of two species $\p$ and $\q$, where $\q[\emptyset]=0$, is defined by
\begin{equation}
(\p\circ\q)[I]:=\bigoplus_{X\vdash I} \p[X] \otimes (\bigotimes_{S\in X} \q[S]).
\end{equation}

A positive species is a species $\p\in\Sp$ such that $\p[\emptyset]=0$. Denote by $\Sp_+$ the subcategory of positive species. We can form a positive species $\p_+$ from any species $\p$, by setting \begin{equation*} \p_+[I] = \begin{cases} \p[I] & \text{if $I\neq \emptyset$,}\\ 0 & \text{otherwise.} \end{cases} \end{equation*}

\begin{prop}
For any positive species $\p\in\Sp_+$,
\begin{equation}
\ell\circ\p=\bigoplus_{n\geq 0} \p^{\cdot n}
\end{equation}
where $\p^{\cdot 0}=\o$. Furthermore, for any finite set $I$,
$$(\ell\circ \p)[I]=\bigoplus_{F\models I} \p(F).$$
\end{prop}
The proof can be found in \cite{Aguiar}.

\begin{definition}
Let
$$T:\Sp_+\rightarrow \mathbf{Hopf}(\Sp)$$
be the functor defined by $T(\p)=\ell\circ \p$. The structure maps, multiplication and comultiplication, of $T(\p)$ are given by {\bf concatenation} and {\bf deshuffle}, respectively, and are defined the following way:
\begin{enumerate}
\item For a finite set $I$, a decomposition $S\sqcup T=I$ and two compositions $F\models S$ and $G\models T$, the component map of the multiplication is the direct sum of the identity maps
$$\p(F)\otimes \p(G)\rightarrow \p(F|G).$$
\item Let $F=F_1|\dots |F_k$ be a composition of $I$. For each decomposition $I=S\sqcup T$ such that $S$ and $T$ are a union of blocks of $F$, the component of the comultiplication is the map
\begin{equation}
\p(F)\rightarrow \p(F|_S)\otimes \p(F|_T).
\end{equation}
\end{enumerate}
\end{definition}

\begin{ex}
Consider the species $\exp$. Let $\exp_+$ be the positive species defined by
\begin{equation} \exp_+[I]=\begin{cases} \exp[I] & \text{if $I\neq \emptyset$,}\\ 0 & \text{otherwise.}\end{cases}\end{equation}
Now, $T(\exp_+)[I]=\bigoplus_{F\models I} \exp_+(F)$ and denote $F_1|F_2|\dots|F_k$ the only element in $\exp_+(F)=\exp_+[F_1]\otimes \dots \otimes \exp_+[F_k]$. For $I=[3]$, $F=1|2$ and $G=3$, the multiplication $\mu$ is given by $\mu((1|2)\otimes 3)= 1|2|3$. And for $1|23\models [3]$, the comultiplication is $\Delta(1|23) = 1|23\otimes \emptyset + 1\otimes 23 + 23\otimes 1 + \emptyset\otimes 1|23$
\end{ex}

\begin{definition}\label{Tvee}
Let
$$T^{\vee}:\Sp_+\rightarrow \mathbf{Hopf}(\Sp)$$
be the functor defined by $T^{\vee}(\p)=\ell^{\ast}\circ \p$. The structure maps, multiplication and comultiplication, are the {\bf shuffle} and the {\bf deconcatenation}, respectively. 
\begin{enumerate}
\item For a finite set $I$, fix a decomposition $S\sqcup T=I$. For each compositions $F\models S$, $G\models T$ the multiplication is the direct sum, over all quasishuffle $H$ of $F$ and $G$, of the unique map
$$\p(F)\otimes\p(G)\rightarrow \p(H)$$
obtained by reordering the factors.
\item For each composition $F_1|\dots |F_k\models I$, the comultiplication is the direct sum, over all decompositions $S\sqcup T=I$ for which $S$ is the union of the first
$i$ blocks, of
$$\p(F)\rightarrow \p(F|_S)\otimes \p(F|_T).$$
\end{enumerate}
\end{definition}

\begin{ex}
Consider the species $T^{\vee}(\exp_+)$. As a vector space it is isomorphic to $T(\exp_+)$ but not as a Hopf monoid. In fact the multiplication $\mu^{\ast}$ and comultiplication $\Delta^{\ast}$ are given, for $F=1|2$ and $G=3$, by $\mu^{\ast}((1|2)\otimes 3)=1|2|3+1|3|2+3|1|2+ 1|23 + 13|2$ and $\Delta^{\ast}(1|2)=1|2\otimes \emptyset + 1\otimes 2+ \emptyset\otimes 1|2$.

\end{ex}

\section{Functors}\label{Sec:functors}
In this section, we recall some useful facts about functors and make explicit the diagram~\ref{Diag:KBt}:
\begin{equation*}
\xymatrix{ & \SpB \ar[dr]^{\KBt} \ar@{=>}[d]^{\alpha} \\ \Sp \ar[rr]_{\Kb} \ar[ur]^{\S} & & \gVec}
\end{equation*}
A vector space $V$ is said to be graded if there is a sequence of vector spaces $V_i$, $i\geq 0$, such that $V=\bigoplus_{i\geq 0} V_i$. A morphism $f\colon V\rightarrow W$ of graded vector spaces is a sequence of morphisms $f_i\colon V_i\rightarrow W_i$, one for each $i$. We write $f=\bigoplus_{i\geq 0} f_i$. Let $\gVec$ be the category with objects and arrows just described. It becomes a monoidal category with the product, on component of degree $i$:
$$(V\cdot W)_i = \bigoplus_{j=0}^i V_j\otimes W_{n-j}$$

Let $F\colon C\rightarrow D$ be a functor between two monoidal categories $(C,\cdot)$ and $(D,\cdot)$ with unit $I$ in both categories. For the next definition, see \cite{Mac} and \cite{Aguiar} for more details.
\begin{definition}
$F$ is a lax tensor functor if there are natural transformations
$$\phi_{A,B}:F(A)\cdot F(B) \rightarrow F(A\cdot B)$$
$$\phi_0:I\rightarrow F(I)$$
such that $\phi$ is associative, in the sense that there two unambiguous maps
$$F(A)\cdot F(B)\cdot F(C) \rightarrow F(A\cdot B\cdot C)$$
and $\phi_0$ is left and right unital. In \cite{Mac}, this is called a monoidal functor.

$F$ is a tensor colax functor if there are natural transformations
$$\psi_{A,B}:F(A\cdot B) \rightarrow F(A)\cdot F(B)$$
$$\psi_0:F(I)\rightarrow I$$
satisfying dual axioms of the lax tensor functor.

$F$ is bilax if $(F,\phi)$ is lax, $(F,\psi)$ is colax and such that they satisfy the braiding condition, i.e. two unambiguous maps 
$$F(A\cdot B)\cdot F(C\cdot D)\rightarrow F(A\cdot C)\cdot F(B\cdot D)$$
and unital conditions.
\end{definition}

The importance of bilax tensor functors is that they preserve the structure of bimonoids.

\begin{definition}
The functor
$$\S:\Sp\rightarrow \SpB$$
is defined, for species $\p$, $\H$-sets $[\n]$ and $[\m]$ and $\H$-bijection $f\colon [\n]\rightarrow [\m]$ by
\begin{equation}
\begin{aligned}
\S\p[\n] &= \bigoplus_{s:[n]\rightarrow [\bar{n},n]} \p[s([n])]\\
\S\p[f] &= \bigoplus_{\sn} \p[f|_{s([n])}]
\end{aligned}
\end{equation}
where $s$ is a section map, i.e. for $i\in [n], s(i)\in \{i,\sigma(i)\}$.
\end{definition}

\begin{prop}
The functor $\S$ is bilax with natural transformations $\phi^{\S}=\id$, $\psi^{\S}=\id$, i.e. $\S(\p\cdot\q)\simeq\S\p\cdot\S\q$.
\end{prop}
The proof is straightforward.

Let $G$ be a group and $V$ be a $G$-module then $V^G:=\{x\in V:gx=x,\,\forall g\in G \}$ is the space of $G$-invariants of $V$, and $V_G:=V/\{x-gx:g\in G,\,x\in V\}$ is the space of $G$-coinvariants of $V$. Aguiar and Mahajan \cite{Aguiar} proved that $K$ and $\Kb$ are bilax tensor functor, where
\begin{equation}
K(\p)=\bigoplus_{n\geq 0} \p[n] \text{ and } \Kb(\p)=\bigoplus_{n\geq 0} \p[n]_{\Sn}
\end{equation}

In our case, three functors $\SpB\rightarrow \gVec$ can be defined:
\begin{definition}
Let
\begin{equation*}
\KB,\KBb,\KBt \colon \SpB\rightarrow \gVec
\end{equation*}
be define, for $\p\in\SpB$, by
\begin{equation}
\begin{aligned}
\KB(\p) &= \bigoplus_{n\geq 0} \p[\bar{n},n]\\
\KBb(\p) &= \bigoplus_{n\geq 0} \p[\bar{n},n]_{\Bn}\\
\KBt(\p) &= \bigoplus_{n\geq 0} \p[\bar{n},n]_{S_n}
\end{aligned}
\end{equation}
\end{definition}

Consider natural transformations $\phi$ and $\psi$:
\begin{equation}
\xymatrix{\KB(\p)\cdot\KB(\q) \ar@<1ex>[r]^-{\phi_{\p,\q}} & \KB(\p\cdot\q) \ar@<1ex>[l]^-{\psi_{\p,\q}}}
\end{equation}
to be the direct sum of these maps:
\begin{equation}
\begin{aligned}
\phi_{\p,\q}=&\id\otimes \q[\can]:\\
&\p[\bar{s},s]\otimes \q[\bar{t},t]\mapsto
\p[\bar{s},s]\otimes \q[ [\overline{s+t},s+t]\backslash [\bar{s},s]]\\
\\
\psi_{\p,\q}=&\p[\st]\otimes \q[\st]:\\
&\p[\bar{S}\cup S]\otimes \q[\bar{T}\cup T] \mapsto
\p[\bar{|S|},|S|]\otimes \q[\bar{|T|},|T|]
\end{aligned}
\end{equation}
where $[\overline{s+t},s+t]\backslash [\bar{s},s]=\{\overline{s+t},\dots,\overline{s+1},s+1,\dots,s+t\}$.
First, we verify that $\phi$ is a natural transformation. Let $\alpha\colon \p\rightarrow\p'$ and $\beta\colon\q\rightarrow\q'$ be two $\H$-morphisms. Then, by fixing $n$ and $m$ we have
\begin{equation}
\begin{CD}
\p[\pmn]\otimes\q[\pmm] @>\id\otimes\q[\can]>> \p[\pmn]\otimes\q[\pmmn_n]\\
@V\alpha_{[\pm n]}\otimes \beta_{[\pm m]}VV @VV\alpha_{[\pm n]}\otimes \beta_{[\pm (n+m)]\backslash[\pm n]}V \\
\p'[\pmn]\otimes\q[\pmm] @>\id\otimes\q'[\can]>> \p'[\pmn]\otimes\q'[\pmmn_n]
\end{CD}
\end{equation}
where $[\pm (n+m)]\backslash[\pm n]=\{\overline{n+m},\dots,\overline{n+1},n+1,\dots,n+m\}$. The diagram commutes since $\alpha$ and $\beta$ are natural transformations and $\can$ is a $\H$-bijection.
We check next that $\phi$ is associative and that $\phi$ and $\psi$ satisfy the braiding condition. All the unitality conditions follow, since all the maps involved are isomorphisms.

{\bf Associativity}.
$$\xymatrix{\KB(\p)\cdot\KB(\q)\cdot\KB(\r) \ar[r] & \KB(\p\cdot\q\cdot\r)}$$
The map $\phi$ leads to an unambiguous map defined by
\begin{equation*}
\xymatrix{\p[\pmn]\otimes\q[\pmm]\otimes\r[\pms] \ar[d]^{\id\otimes\q[\can]\otimes\r[\can]}\\ \p[\pmn]\otimes \q[\pmmn_n]\otimes\r[\pmmns_mn]}
\end{equation*}
%and
%\begin{equation*}
%\xymatrix{\p[\pmS]\otimes\q[\pmT]\otimes\r[\pmU] \ar[d]^{\p[\st]\otimes\q[\st]\otimes\r[\st]}\\ \p[\pmSs]\otimes \q[\pmTt]\otimes\r[\pmUu]}
%\end{equation*}

{\bf Braiding}.
We must show that both directions for
$$\KB(\p\cdot\q)\cdot\KB(\r\cdot\s)\rightarrow\KB(\p\cdot\r)\cdot\KB(\q\cdot\s)$$
lead to unambiguous maps. Let $[\pm S]\sqcup [\pm T]=[\pm n]$ and $[\pm U]\sqcup [\pm V]=[\pm m]$. Taking the direct over these sets, the maps $\st$ and $\can$ induce an isomorphism:
$$\xymatrix{(\p[\pmS]\otimes\q[\pmT])\otimes(\r[\pmU]\otimes\s[\pmV])\ar[d] \\ 
\p[\pms[\otimes\r[[\pm (s+u)]\backslash [\pms]])\otimes (\q[\pmt]\otimes \s[[\pm (t+v)]\backslash [\pmt]]}$$

We have just given an idea for $\phi$ and the same can be done for $\psi$ using the dual axioms, for the proof of:

%$\KB(\p)\cdot\KB(\q)$ and $\KB(\p\cdot\q)$, where the domain and codomain are
%$$\KB(\p)\cdot\KB(\q)=\bigoplus_{n\geq 0}\bigoplus_{s+t=n} \p[([\bar{s},s],\sigma_0)]\otimes \q[([\bar{t},t],\sigma_0)]$$
%\begin{equation}
%\begin{aligned}
%\KB(\p\cdot \q)&=\bigoplus_{n\geq 0}(\p\cdot \q)[[\bar{n},n],\sigma_0]\\
%&= \bigoplus_{n\geq 0}\bigoplus_{S\sqcup T=[n]} \p[\bar{S}\cup S,\sigma_0|_S]\otimes \q[\bar{T}\cup T,\sigma_0|_T]
%\end{aligned}
%\end{equation}

\begin{prop}
The functor $\KB$ is bilax. 
\end{prop}

%With this in mind, it make sense to talk about the space of $B_n$-invariants and the space of $B_n$-coinvariants of
%a vector space $\p[[\bar{n},n],\sigma_0]$ coming from a $\H$-species $\p$.

\begin{prop}
$\KBt$ is a bilax tensor functor.
\end{prop}
\begin{proof}[\bf Proof.] 
The natural transformations $\widetilde{\phi}_{\p,\q}=\id\otimes\q[\can]$ and $\widetilde{\psi}_{\p,\q}=\p[\st]\otimes\q[\st]$ are the same as for $\KB$. We prove that $\psi$ is well-defined, and the same can be done for $\phi$. In other words, this diagram has to commute:
\begin{equation}
\xymatrix{\KB(\p\cdot \q) \ar[d] \ar[r]^-{\psi_{\p,\q}}  &  \KB(\p)\cdot \KB(\q)\ar[d]\\
 \KBt(\p\cdot \q)\ar[r]^-{\widetilde{\psi_{\p,\q}}}  & \KBt(\p)\cdot \KBt(\q)}
\end{equation}
By fixing $n$, we have
\begin{equation}
\xymatrixcolsep{5pc}\xymatrix{ {\bigoplus \atop S\sqcup T=[n]} \p[\pm S]\otimes\q[\pm T] \ar[d] \ar[r]^-{\p[\st]\otimes\q[\st]}  &  {\bigoplus \atop s+t=n} \p[\pms[\otimes \q[\pmt]\ar[d]\\
\bigl{(}{\bigoplus \atop S\sqcup T=[n]} \p[\pm S]\otimes\q[\pm T]\bigr{)}_{\Sn} \ar[r]_-{\p[\st]\otimes \q[\st]} & {\bigoplus \atop s+t=n} \p[\pms]_{S_s} \otimes \q[\pmt]_{S_t}}
\end{equation}

Let $\tau\in \Sn$. For a decomposition $S|T\models [n]$, suppose that $\tau(S)=U$ and that $\tau(T)=V$, then $\tau$ send $-S$ to $-U$ and $-T$ to $-V$. So $\tau$ give rise to two permutations $\tau_1\colon [\pm S]\rightarrow [\pm U]$ and $\tau_2\colon [\pm T]\rightarrow [\pm V]$, and again to two others $\tau_s:=\st\circ\tau_1\circ \st^{-1}\colon [\pm |S|]\rightarrow [\pm |U|]$ and $\tau_t:=\st\circ\tau_2\circ\st^{-1}\colon [\pm |T|]\rightarrow [\pm |V|]$, so that
\begin{equation}
\xymatrixcolsep{5pc}\xymatrix{ \p[\pm S]\otimes\q[\pm T] \ar[d]_{\p[\tau_1]\otimes \q[\tau_2]} \ar[r]^-{\p[\st]\otimes\q[\st]}  &  \p[\pm |S| ]\otimes \q [\pm |T|]      \ar[d]^{\p[\tau_s]\otimes\q[\tau_t]} \\
p[\pm U]\otimes\q[\pm V] \ar[r]_-{\p[\st]\otimes \q[\st]} & \p[\pm |S|] \otimes \q[\pm |T|]}
\end{equation}
commutes since $\psi$ is a natural transformation. It insures that $\phi$ and $\psi$ factor through coinvariants.
\end{proof}

Since $\KB$ and $\KBt$ are bilax tensor functors, for any bimonoid $\p$, $\KB(\p)$ and $\KBt(\p)$ are graded bialgebras. As an example, consider $\KB(\F)$, where $\F$ is the $\H$-species of section maps. Let $F=\{s_1,\dots,s_k\}$ be a set in $\F[\pm n]$. So each $s_j$ is a section map. Define 
$$F^+:=\{i: s_j(i)=i,\, 1\leq j\leq k\} \text{ , } F^-:= \{i: s_j(i)=\bar{i},\, 1\leq j\leq k\} \text{ and } F^{\ast}=[n]\backslash(F^+\cup F^-)$$
Now, encode $F$ by a word $w_F$ in three letters $a,b,c$, where the ith-letter of $w_F$ is
\begin{equation*}
w_F(i)= \begin{cases}  a & \text{if $i\in F^+$,}\\ b & \text{if $i\in F^-$,}\\ c & \text{otherwise.} \end{cases}
\end{equation*}
For example, the word associated to $F=\{(1,2,\bar{3},4),(1,\bar{2},\bar{3},4),(1,2,\bar{3},\bar{4}),(1,\bar{2},\bar{3},\bar{4})\}$ is $w_F=acbc$ and the one for $G=\{(5,\bar{6})\}$ is $w_G=ab$. Let $H=\mu_{\F}(F,G)$. Since $H^+=F^+\cup G^+$ and $H^-=F^-\cup G^-$, the multiplication is given by the concatenation of the words: $\mu_{\F}(acbc,ab)=acbcab$. Let $\Delta_{\F}(H)=\sum H_1 \otimes H_2$, where the sum is over $H_1^+,H_2^+,H_1^-,H_2^-,H_1^{\ast},H_2^{\ast}$, where $H_1^+\cup H_2^+ = H^+$, $H_1^-\cup H_2^-=H^-$ and $H_1^{\ast}\cup H_2^{\ast}=H^{\ast}$. Translated into words, the comultiplication is the deshuffle $\Delta(ab)=ab\otimes \emptyset + a\otimes b + b\otimes a + \emptyset\otimes ab$. So $\KB(\F)$ is isomorphic to the algebra freely generated by $\{a,b,c\}$.

%Let $\p$ be a monoid in $\SpB$ with multiplication $\mu$ and let $x\in\p[([n,\bar{n}],\sigma_0)]$ and $y\in\p[([m,\bar{m}],\sigma_0)]$. Define the multiplication of $x$ and $y$ in $\KB(\p)$ as $x\ast y=\mu(x,\can(y))$ where $\can\colon [m,\bar{m}]\rightarrow [\overline{n+m},n+m]\backslash [\bar{n},n]$. Now, if $\q$ is a comonoid in $\SpB$ with comultiplication $\Delta$ and $x\in\q[([n,\bar{n}],\sigma_0)]$, with $\Delta(x)=\sum x_{(1)}\otimes x_{(2)}$ then define the comultiplication in $\KB(\p)$ as  $x\mapsto \sum \st(x_{(1)})\otimes \st(x_{(2)})$.
%where for any decomposition $S|T\models \n$, the standardization maps 
%where if $x_{(1)}\in \q[(\overline{S_1}\cup S_1,\sigma)]$ and $x_{(2)}\in \q[(\overline{S_2}\cup S_2,\sigma)]$, $S_1\sqcup S_2=[n]$, then $\st:\overline{S_1}\cup S_1 \rightarrow [\overline{|S_1|},|S_1|]$ and $\st:\overline{S_2}\cup S_2 \rightarrow [\overline{|S_2|},|S_2|]$.

We are now interested in the composition of the two functors $\S\colon \Sp\rightarrow \SpB$ and $\KBt\colon \SpB\rightarrow \gVec$, $\KBt\circ\S$, which is bilax since both functors are bilax, and in the relation with the functor $\Kb$, i.e. the natural transformation $\widetilde{\alpha^{\S}}\colon \KBt\circ\S\rightarrow \Kb$. First, let us look at an example of the composition:

\begin{ex}
$$\KBt\circ\S(\eB)=\K\oplus \K[1]\oplus \K[\bar{1}]\oplus \K[12]\oplus \K[1\bar{2}]\oplus \K[\bar{1}\bar{2}]\oplus \dots$$
Now, let $S$ and $T$ be two elements of $\KBt\circ\S(\eB)$. The multiplication is given by $\mu(S\otimes T)=\mu(S\otimes \can(T))$ and the comultiplication is given by $\Delta(S)=\sum_{I=S\sqcup T} \st(S)\otimes \st(T)$. So, as a bialgebra $\KBt\circ\S(\eB)$ is isomorphic to $\K[x,y]$ by identifying a set $S$ to a monomial $x_S$ where the number of $x$'s and $y$'s in $x_S$ are the number of positive and negative integers in $S$ respectively. On the other hand $\KB\circ\S(\eB)\simeq \K\langle x,y\rangle$ and $\KBb\circ\S(\eB)\simeq \K[x]$.
\end{ex}

%Now, we would like to compare the composition $\KB\circ\S$ with $K$, i.e. find the natural transformation $\alpha^S\colon \KB\circ\S\rightarrow K$.
%\begin{equation}
%\xymatrix{& \SpB \ar@{=>}[d]^{\alpha^{\S}} \ar[rd]^{\KB} \\ \Sp \ar[ru]^{\S} \ar[rr]_{K} & & \gVec }
%\end{equation}

\begin{prop}
The natural transformation $\widetilde{\alpha^{\S}}:\KBt\circ \S\rightarrow \Kb$ is defined by, for any $\p\in\Sp$,
\begin{equation}
\widetilde{\alpha^{\S}}_{\p}=\bigoplus_{n\geq 0 \atop \sn}\p[s^{-1}]
\end{equation}
where $s^{-1}\colon s([n])\rightarrow [n]$ is the unique bijection.
\end{prop}
\begin{proof}[\bf Proof.]
$\widetilde{\alpha^{\S}}$ assigns a linear map $\widetilde{\alpha^{\S}}_{\p}$ for every object $\p\in\Sp$. Furthermore, for all arrow $\beta:\p\rightarrow \q$ in $\Sp$, we must prove that this diagram commutes:
\begin{equation}
\begin{CD}
\KBt\circ\S(\p) @>\bigoplus_{n\geq 0 \atop \sn}\p[s^{-1}]>> \Kb(\p) \\
@V (\KBt\circ\S)[\beta] VV @VVK[\beta]V\\
\KBt\circ\S(\q) @>>\bigoplus_{n\geq 0 \atop \sn}\q[s^{-1}]> \Kb(\q)
\end{CD}
\end{equation}
By fixing $n$, we get the following diagram for each section map:
\begin{equation}
\begin{CD}
\bigoplus_{\sn}\p[s([n])] @>\bigoplus_{\sn}\p[s^{-1}]>> \p[n] \\
@V \bigoplus_{\sn}\beta_{s([n])} VV @VV\beta_nV\\
\bigoplus_{\sn}\q[s([n])] @>>\bigoplus_{\sn}\q[s^{-1}]> \q[n]
\end{CD}
\end{equation}
Since $\beta$ is a natural transformation and $s^{-1}\colon s([n])\rightarrow [n]$ is a bijection, this diagram commutes.
\end{proof}

\begin{ex}
Consider the linear order species. On one hand
\begin{equation}
\KBt\circ \S(\ell) = \K\oplus \K[1]\oplus \K[\bar{1}]\oplus \K[12]\oplus \K[\bar{1}2]\oplus \K[1\bar{2}] \oplus \K[\bar{1}\bar{2}]\oplus \dots
\end{equation}
and on the other hand
\begin{equation}
\Kb(\ell)=\K\oplus\K[1]\oplus \K[12] \oplus \K[123] \oplus\dots
\end{equation}

Now applying $\widetilde{\alpha^{\S}}_{\ell}$ to $\KBt\circ \S (\ell)$, we see on this example that $\widetilde{\alpha^{\S}}$ is the transformation between a $B_n$-module
and a $S_n$-module that forgets the sign.
\end{ex}

%\begin{rem}
%For a species $\p\in\Sp$, the natural transformation $\widetilde{\alpha^{\S}}:\KBt\circ\S\rightarrow\Kb$ is $\p[s^{-1}]$ and $\overline{\alpha^{\S}}:\KBb\circ \S %\rightarrow \Kb$ is $\p[\id]$. 
%\end{rem}

\section{Hopf algebra of set compositions}\label{Sec:QPi}

As a final word, we study the species $\ell^{\ast}\circ \exp_+$ and its algebraic operations under the composite of functors
$$\Sp \buildrel \S \over{\longrightarrow} \SpB \buildrel \KBt \over{\longrightarrow}\gVec$$
and we give the isomorphism $\KBt\circ \S(\ell^{\ast}\circ\exp_+)\simeq \DQSym$. Let $\DQL=\KBt\circ \S(\ell^{\ast}\circ\exp_+)$. Explicitely,
\begin{equation}
\DQL:=\bigoplus_{n\geq 0} \Bigl{(}\bigoplus_{s:[n]\rightarrow [\bar{n},n]} \bigoplus_{F\models s([n])}\exp_+(F)\Bigr{)}_{S_n}\\
\end{equation}
Recall that we denote $F_1|F_2|\dots |F_l$, the only element in $\exp_+[F_1]\otimes \exp_+[F_2]\otimes\dots\otimes\exp_+[F_l]$. Now, the representative of each equivalence class will be the set composition in which the integers are ordered in absolute value and where all the negative integers are on the left of each part. As an example, $\bar{1}2|\bar{3}$ is the representative of $[\bar{1}2|\bar{3},\bar{2}1|\bar{3},\bar{1}3|\bar{2},\bar{3}1|\bar{2},\bar{2}3|\bar{1},\bar{3}2|\bar{1}]$.

\begin{definition}
A bicomposition $\Bigl{(}{\alpha_1 \atop \beta_1}\dots {\alpha_k \atop \beta_k}\Bigr{)}$ of $n$ is a list of $1\times 2$ vectors $\Bigl{(}{\alpha_i \atop \beta_i}\Bigr{)}\neq \Bigl{(} {0\atop 0}\Bigr{)}$ called biparts such that $\sum_{i=1}^k(\alpha_i+\beta_i)=n$.
\end{definition}
For example, the bicompositions of $n=2$ are
$$\Bigl{(}{2 \atop 0}\Bigr{)}, \Bigl{(}{1 \atop 0}{1\atop 0}\Bigr{)}, \Bigl{(}{1 \atop 1}\Bigr{)}, \Bigl{(}{1 \atop 0}{0\atop 1}\Bigr{)}, \Bigl{(}{0 \atop 1}{1\atop 0}\Bigr{)}, \Bigl{(}{0 \atop 1}{0\atop 1}\Bigr{)}, \Bigl{(}{0 \atop 2}\Bigr{)}$$

\begin{prop}
Each equivalence class of $A$ is determined by the number of positive and negative integers in 
each part. In other words, the elements of $A$ are in bijection with the bicompositions.
\end{prop}
\begin{proof}[\bf Proof.]
For a class $[x]$ and for any element $x\in[x]$, the number of positive integers in the i-th part will be denoted by $|x_i|^+$ and the number of negative integers in the i-th part will denoted by $|x_i|^-$. Let $[x]$ and $[y]$ be two elements of $A$, with $x\in [x]$ and $y\in [y]$. Suppose that for all $i$, $|x_i|^+=|y_i|^+$ and $|x_i|^-=|y_i|^-$. Let $x_i^j$ and $y_i^j$ be the $j$-th integers in the $i$-th part of $x$ and $y$ respectively. Build a permutation $\pi$ by setting $\pi(|x_i^j|)=|y_i^j|$, for all $i,j$. So $x$ and $y$ are in the same class.

Now, form a bicomposition $\Bigl{(}{\alpha_1 \atop \beta_1}\dots {\alpha_k \atop \beta_k}\Bigr{)}$, from a composition $x$ of $n$ integers, by setting $\alpha_i$ and $\beta_i$ to be respectively the number of negative and positive integers in the i-th part of $x$. It is a bicomposition as each vector is non-zero and the sum of its integers is $n$.

Conversely, form a composition $x$ from a bicomposition, $\Bigl{(}{\alpha_1 \atop \beta_1}\dots {\alpha_k \atop \beta_k}\Bigr{)}$, by building each part of $x$, where the ith-part is given by the negative integers:
$$\overline{\alpha_1+\beta_1+\dots+ \alpha_{i-1}+\beta_{i-1}+1}, \dots ,\overline{\alpha_1+\beta_1+\dots+ \alpha_{i-1}+\beta_{i-1}+\alpha_i}$$
and the positive integers:
$$(\alpha_1+\beta_1+\dots+ \alpha_{i}+1),\dots ,(\alpha_1+\beta_1+\dots+ \alpha_{i}+\beta_i)$$

This set composition is part of a class of $A$ with $\alpha_i$ negative integers and $\beta_i$ positive integers in the $i$-th part. So for every bicomposition, there is a unique class of $A$ associated to it.
\end{proof}

For example, the bicomposition $\Bigl{(}{0\atop 2} {1\atop 0}\Bigr{)}$ represents the class $[12|\bar{3},13|\bar{2},23|\bar{1}]$. As we have seen in Definition~\ref{Tvee}, the product is given by quasishuffle and the coproduct by deconcatenation. For example,
$\mu(\bar{1}\bar{2},\bar{1}|2) = \mu(\bar{1}\bar{2},\bar{3}|4) = \bar{1}\bar{2}|\bar{3}|4 + \bar{1}|\bar{2}\bar{3}|4 + \bar{1}|2|\bar{3}\bar{4}+ \bar{1}\bar{2}\bar{3}|4 + \bar{1}|\bar{2}\bar{3}4$ which correspond to the quasishuffle of the bicomposition by the shuffle of the columns and two adjacent column coming from different composition can be added using: $ \bigl{(}{\alpha_i \atop \beta_i}\bigr{)} +\bigl{(}{\alpha_{i+1} \atop \beta_{i+1}}\bigr{)} =\bigl{(}{\alpha_i + \alpha_{i+1} \atop \beta_i + \beta_{i+1}}\bigr{)}$.
For example,
\begin{equation}
\mu( \Bigl{(} {2\atop 0} \Bigr{)} ,\Bigl{(} {1\atop 0} {0\atop 1}\Bigr{)})= \Bigl{(} {2\atop 0} {1\atop 0} {0\atop 1}\Bigr{)}+\Bigl{(} {1\atop 0} {2\atop 0}{0\atop 1}\Bigr{)}+\Bigl{(} {1\atop 0} {0\atop 1}{2\atop 0}\Bigr{)}+\Bigl{(} {3\atop 0} {0\atop 1}\Bigr{)}+\Bigl{(} {1\atop 0} {2\atop 1}\Bigr{)}
\end{equation}

An example of the comultiplication $\Delta(\bar{1}|\bar{2}3)= \bar{1}|\bar{2}3\otimes \emptyset +\bar{1}\otimes \bar{1}2+
\emptyset\otimes \bar{1}|\bar{2}3$ or equivalently, on bicomposition
$$\Delta({ 1\: 1 \choose 0\: 1})= { 1\: 1 \choose 0\: 1} \otimes \emptyset + {1\choose 0}\otimes {1\choose 1} + \emptyset \otimes { 1\: 1 \choose 0\: 1}$$

\begin{definition}
$\DQSym$ is the Hopf algebra generated by $\{M_{ \bigl{(}{\alpha \atop\beta} \bigr{)} } : \bigl{(}{\alpha \atop\beta} \bigr{)} \text{ bicomposition}\}$. The multiplication and the comultiplication are given respectively by quasishuffle and deconcatenation.
\end{definition}
See \cite{AB} for more detailed information about this algebra.

\begin{theorem}
\begin{equation}
\DQL \simeq \DQSym \text{ as Hopf algebras}
\end{equation}
\end{theorem}

\begin{proof}[\bf Proof.] 
The isomorphism sending a diagonal composition $\Bigl{(} {\alpha \atop \beta} \Bigr{)}=\Bigl{(} {\alpha_1\atop\beta_1} {\alpha_2\atop\beta_2} \dots {\alpha_k \atop\beta_k}\Bigr{)}$ to the monomial basis $M_{\bigl{(} {\alpha \atop \beta} \bigr{)}}$ is one of Hopf algebra as it preserves product and coproduct.
\end{proof}

\end{document}